\documentclass[12pt,reqno]{amsart}
\usepackage{amsmath,amssymb}
\setbox0=\hbox{$+$}
\newdimen\plusheight
\plusheight=\ht0
\def\+{\;\lower\plusheight\hbox{$+$}\;}

\setbox0=\hbox{$-$}
\newdimen\minusheight
\minusheight=\ht0
\def\-{\;\lower\minusheight\hbox{$-$}\;}

\setbox0=\hbox{$\cdots$}
\newdimen\cdotsheight
\cdotsheight=\plusheight
\def\cds{\lower\cdotsheight\hbox{$\cdots$}}

\renewcommand{\(}{\left\(}
\renewcommand{\)}{\right\)}
\renewcommand{\[}{\left[}

 \theoremstyle{plain}
\newtheorem{theorem}{Theorem}[section]
\newtheorem{lemma}[theorem]{Lemma}

\newtheorem{corollary}[theorem]{Corollary}
\newtheorem{definition}{Definition}

\newtheorem{remark}[theorem]{Remark}

\newenvironment{pf}
   {\vskip 0.15in \par\noindent{\bf Proof of Theorem}\hskip 0.5em\ignorespaces}
   {\hfill $\Box$\par\medskip}
\begin{document}
\title[Hypergeometric functions and a family of algebraic curves]
{Hypergeometric functions and a family of algebraic curves}
\author{Rupam Barman}
\address{Department of Mathematical Sciences, Tezpur University, Napaam-784028, Sonitpur, Assam, INDIA}
\email{rupamb@tezu.ernet.in}
\author{Gautam Kalita}
\address{Department of Mathematical Sciences, Tezpur University, Napaam-784028, Sonitpur, Assam, INDIA}
\email{gautamk@tezu.ernet.in}
\vspace*{0.7in}
\begin{center}

{\bf HYPERGEOMETRIC FUNCTIONS \\AND A FAMILY OF ALGEBRAIC CURVES}\\[5mm]

Rupam Barman and Gautam Kalita\\[.2cm]
\end{center}
\vspace{.51cm}
\noindent\textbf{Abstract:} Let $\lambda \in \mathbb{Q}\setminus \{0, 1\}$ and $l \geq 2$, and denote by
$C_{l,\lambda}$ the nonsingular projective algebraic curve over $\mathbb{Q}$ with affine equation given by
$$y^l=x(x-1)(x-\lambda).$$
In this paper we define $\Omega(C_{l, \lambda})$ analogous to the real periods of elliptic curves and find a
relation with ordinary hypergeometric series. We also give a relation between the number of points on $C_{l, \lambda}$
over a finite field and Gaussian hypergeometric series. Finally we give an alternate proof of a result of \cite{rouse}.

\noindent{\footnotesize \textbf{Key Words}: algebraic curves; hypergeometric series.}

\noindent{\footnotesize 2010 Mathematics Classification Numbers: 11G20, 33C20}

\section{\bf Introduction}\label{secone}

Hypergeometric functions and their relations with algebraic curves have been studied by many mathematicians.
For $a_0, a_1, \ldots, a_r, b_1, b_2, \ldots, b_r \in \mathbb{C}$, the ordinary hypergeometric series ${_{r+1}}F_r$
is defined as
$${_{r+1}}F_r\left(\begin{array}{cccc}
                a_0, & a_1, & \cdots, & a_r\\
                 & b_1, & \cdots, & b_r
              \end{array}\mid z \right):
              =\sum_{n=0}^{\infty}\frac{(a_0)_n(a_1)_n\cdots (a_r)_n}{(b_1)_n(b_2)_n\cdots (b_r)_n}\frac{z^n}{n!},$$
where $(a)_0=1$, $(a)_n:=a(a+1)(a+2)\ldots(a+n-1)$ for $n\geq 1$, and none of the $b_i$ is a negative integer or zero.
This hypergeometric series converges absolutely for $|z|<1$. The series also converges absolutely for
$|z|=1$ if $\text{Re}(\sum b_i - \sum a_i)> 0$ and converges conditionally for $|z|=1$, $z\neq 1$
if $0\geq \text{Re}(\sum b_i - \sum a_i)> -1$. For details see \cite[chapter 2]{andrews}.

In \cite{greene} Greene introduced the notion of Gaussian hypergeometric series over finite fields.
Since then, the interplay between ordinary hypergeometric series and Gaussian hypergeometric series has played
an important role in character sum evaluation \cite{greenestanton}, the representation theory of
$SL(2, \mathbb{R})$ \cite{greene1} and finding the number of points on an algebraic curve over finite
fields \cite{ono}. Recently, J. Rouse \cite{rouse} and D. McCarthy \cite{mccarthy} provided an expression for
the real period of certain families of elliptic curves in terms of ordinary hypergeometric series. They also
provided an analogous expression for the trace of Frobenius of the same family of curves in terms of Gaussian
hypergeometric series, which developed the interplay between the two hypergeometric series more fully.

We will now restate some definitions from \cite{greene} which are analogous to the binomial coefficient and ordinary
hypergeometric series respectively. Throughout the paper $p$ is an odd prime. We also let $\mathbb{F}_p$ denote the
finite field with $p$ elements and we extend all characters $\chi$ of $\mathbb{F}_p^\times$ to $\mathbb{F}_p$ by
setting $\chi(0)=0$. For characters $A$ and $B$ of $\mathbb{F}_p$, define ${A \choose B}$ as
\begin{align}\label{eq0}
{A \choose B}:=\frac{B(-1)}{p}J(A,\overline{B})=\frac{B(-1)}{p}\sum_{x \in \mathbb{F}_p}A(x)\overline{B}(1-x),
\end{align}
where $J(A, B)$ is the Jacobi sum of the characters $A$ and $B$ of $\mathbb{F}_p$ and $\overline{B}$ is the
inverse of $B$. With this notation, for characters $A_0, A_1,\ldots, A_n$ and $B_1, B_2,\ldots, B_n$ of $\mathbb{F}_p$,
the Gaussian hypergeometric series ${_{n+1}}F_n\left(\begin{array}{cccc}
                A_0, & A_1, & \cdots, & A_n\\
                 & B_1, & \cdots, & B_n
              \end{array}\mid x \right)$ over $\mathbb{F}_p$ is defined as
\begin{align}\label{eq00}
{_{n+1}}F_n\left(\begin{array}{cccc}
                A_0, & A_1, & \cdots, & A_n\\
                 & B_1, & \cdots, & B_n
              \end{array}\mid x \right):
              =\frac{p}{p-1}\sum_{\chi}{A_0\chi \choose \chi}{A_1\chi \choose B_1\chi}
              \cdots {A_n\chi \choose B_n\chi}\chi(x),
\end{align}
where the sum is over all characters $\chi$ of $\mathbb{F}_p$.

Let $\lambda \in \mathbb{Q}\setminus\{0, 1\}$ and $l \geq 2$, and denote by $C_{l,\lambda}$
the nonsingular projective algebraic curve over $\mathbb{Q}$ with affine equation given by
\begin{align}\label{curve1}
y^l=x(x-1)(x-\lambda).
\end{align}
The change of variables $(x, y) \mapsto (x+\frac{1+\lambda}{3}, \frac{y}{2})$ takes \eqref{curve1} to
\begin{align}\label{curve2}
y^l=2^l(x-a)(x-b)(x-c),
\end{align}
where $a=-\frac{1+\lambda}{3}$, $b=\frac{2\lambda-1}{3}$, and $c=\frac{2-\lambda}{3}$.

We now define an integral for the family of curves \eqref{curve1} analogous to the real period of elliptic curves.
\begin{definition}
The complex number $\Omega(C_{l,\lambda})$ is defined as
\begin{align}\label{definition1}
\Omega(C_{l, \lambda}):=2\int_a^b\frac{dx}{y^{l-1}},
\end{align}
where $x$ and $y$ are related as in \eqref{curve2}.
\end{definition}
\begin{definition}
Suppose $p$ is a prime of good reduction for $C_{l, \lambda}$. Define the integer $a_p(C_{l, \lambda})$ by
\begin{align}\label{eq45}
a_p(C_{l, \lambda}):=1+p-\#C_{l, \lambda}(\mathbb{F}_p),
\end{align}
where $\#C_{l, \lambda}(\mathbb{F}_p)$ denotes the number of points that the curve $C_{l, \lambda}$ has over
$\mathbb{F}_p$.
\end{definition}
It is clear that a prime $p$ not dividing $l$ is of good reduction for $C_{l, \lambda}$ if and only if
$\text{ord}_p(\lambda(\lambda-1))=0.$
\begin{remark}
Let $l \neq 3$. Then $$\#C_{l, \lambda}(\mathbb{F}_p)=1+\#\{(x, y)\in \mathbb{F}_p^2: y^l=x(x-1)(x-\lambda)\}.$$
Indeed, for $l\geq 4$, the point $[1:0:0]$ is the only point at infinity. Similarly, if $l=2$,
the point at infinity is $[0:1:0]$.

Let $l=3$ and $p\equiv 1$ $($\emph{mod} $3)$. Let $\omega \in \mathbb{F}_p^{\times}$ be of order $3$.
Then there are three points at infinity, namely, $[1:1:0], [1:\omega: 0],$ and $[1:\omega^2:0]$.
Hence, in this case, $$\#C_{l, \lambda}(\mathbb{F}_p)=3+\#\{(x, y)\in \mathbb{F}_p^2: y^l=x(x-1)(x-\lambda)\}.$$

Again, if $l=3$ and $p\equiv 2$ $($\emph{mod} $3)$, then the point at infinity is $[1:1:0]$.
\end{remark}
\begin{remark}\label{lem9}
If $l=3$, $C_{l, \lambda}$ is an elliptic curve. Dehomogenizing the projective curve
$C_{3, \lambda}: Y^3=X(X-Z)(X-\lambda Z)$ by putting $X=1$ and then making the substitution
$$Y\rightarrow \lambda x, Z\rightarrow \lambda\left(y+\dfrac{1+\lambda}{2\lambda^2}\right),$$
we find that $C_{3, \lambda}$ is isomorphic over $\mathbb{Q}$ to the elliptic curve
\begin{align}\label{curve3}
y^2=x^3+\left(\frac{\lambda - 1}{2\lambda^2}\right)^2.
\end{align}
\end{remark}
\begin{remark} If $l=2$, then equation \eqref{curve1} gives an elliptic curve in
Legendre normal form with real period $\Omega(C_{2, \lambda})$, and $a_p(C_{2, \lambda})$
is the trace of the Frobenius endomorphism on the curve over $\mathbb{F}_p$.
\end{remark}

We now recall two results relating $\Omega(C_{2, \lambda})$ and $a_p(C_{2, \lambda})$ to hypergeometric series.
\begin{theorem}\cite{housemoller,rouse}\label{theorem1}
If $0<\lambda<1$, then the real period $\Omega(C_{2, \lambda})$ satisfies
$$\frac{\Omega(C_{2, \lambda})}{\pi}={_{2}}F_1\left(\begin{array}{cccc}
                1/2, & 1/2\\
                 & 1
              \end{array}\mid \lambda \right).$$
\end{theorem}
\begin{theorem}\cite{koike, ono}\label{theorem2}
If \emph{ord}$_p(\lambda(\lambda-1))=0$, then
$$-\frac{\phi(-1)a_p(C_{2, \lambda})}{p}={_{2}}F_1\left(\begin{array}{cccc}
                \phi, & \phi\\
                 & \epsilon
              \end{array}\mid \lambda \right),$$
where $\phi$ and $\epsilon$ are the quadratic and trivial characters of $\mathbb{F}_p$ respectively.
\end{theorem}
In this paper we generalize these results to the algebraic curves $C_{l, \lambda}$.
The aim of this paper is to prove the following main results.
\begin{theorem}\label{theorem3}
If $0< \lambda<1$, then $\Omega(C_{l, \lambda})$ is given by
$$\Omega(C_{l, \lambda}) =\frac{(\Gamma(\frac{1}{l}))^2}{2^{l-2}\lambda^{\frac{l-2}{l}}
\Gamma(\frac{2}{l})}\cdot{_{2}}F_1\left(\begin{array}{cccc}
                (l-1)/l, & 1/l\\
                 & 2/l
              \end{array}\mid \lambda \right).$$
\end{theorem}
\begin{theorem}\label{theorem4}
If $p\equiv 1$ $($\emph{mod} $l)$ and \emph{ord}$_p(\lambda(\lambda-1))=0$, then $a_p(C_{l, \lambda})$ satisfies
$$ -a_p(C_{l, \lambda})=\left\{
\begin{array}{ll}
p\cdot\displaystyle\sum_{i=1}^{l-1}\chi^i(-\lambda^2){_{2}}F_1\left(\begin{array}{cccc}
                \overline{\chi^i}, & \chi^i\\
                 & \chi^{2i}
              \end{array}\mid \lambda \right), & \hbox{if $l\neq 3$;} \\
              2+ p\cdot\displaystyle\sum_{i=1}^{l-1}\chi^i(-\lambda^2){_{2}}F_1\left(\begin{array}{cccc}
                \overline{\chi^i}, & \chi^i\\
                 & \chi^{2i}
              \end{array}\mid \lambda \right), & \hbox{if $l=3$,}
                                    \end{array}
                                  \right.$$
where $\chi$ is a character of $\mathbb{F}_p$ of order $l$.
\end{theorem}
\begin{theorem}\label{theorem5}
For $\lambda=\frac{1}{2}$, we have
\begin{align}\label{eq1/2}
\frac{2^{\frac{(l-3)(l-1)}{l}}\Gamma(\frac{2}{l})}{(\Gamma(\frac{1}{l}))^2}\cdot\Omega(C_{l, \lambda})
=\frac{\displaystyle \binom{\frac{1}{2l}}{\frac{1}{l}}}{\displaystyle \binom{\frac{3-2l}{2l}}{\frac{2-l}{l}}}.
\end{align}
Moreover, if $p\equiv 1$ $($\emph{mod} $l)$, then
\begin{align}\label{eq1/22}
-a_p(C_{l, \lambda})=\left\{\begin{array}{lll}
p\cdot\displaystyle \sum_{i=1}^{\lfloor\frac{l-1}{2}\rfloor}\chi^{-2i}(8)
\left[\displaystyle \binom{\chi^i}{\chi^{-2i}}+{\phi\chi^i \choose \chi^{-2i}}\right],
& \hbox{if $\frac{p-1}{l}$ is odd and $l\neq 3$;} \\
p\cdot\displaystyle\sum_{i=1}^{l-1}\chi^{-i}(8)
\left[\displaystyle \binom{\sqrt{\chi^i}}{\chi^{-i}}+{\phi\sqrt{\chi^i} \choose \chi^{-i}}\right],
& \hbox{if $\frac{p-1}{l}$ is even and $l\neq 3$;}\\
2+p\cdot\displaystyle\sum_{i=1}^{2}
\left[\displaystyle \binom{\sqrt{\chi^i}}{\chi^{-i}}+{\phi\sqrt{\chi^i} \choose \chi^{-i}}\right], & \hbox{if $l= 3$,}
\end{array}
\right.
\end{align}
where $\chi$ is a character of $\mathbb{F}_p$ of order $l$ and $\phi$ is the quadratic character.
\end{theorem}
Here we extend the definition of binomial coefficient to include rational arguments via
$$\displaystyle \binom{n}{k}=\dfrac{\Gamma(n+1)}{\Gamma(k+1)\Gamma(n-k+1)}.$$

We also give a simple proof of the following result of J. Rouse (note that ${1/4 \choose 1/2}$ is real).
\begin{theorem}\label{theorem6}\cite[Theorem 3]{rouse} If $\lambda=1/2$, then
 $$\frac{\sqrt{2}}{2\pi}\cdot \Omega(C_{2, \lambda})={1/4 \choose 1/2}.$$ If $p\equiv 1$ $($\emph{mod} $4)$,
 then $$\frac{-\phi(-2)}{2p}\cdot a_p(C_{2, \lambda})=\emph{Re}{\chi_4 \choose \phi},$$
where $\chi_4$ is a character on $\mathbb{F}_p$ of order $4$ and $\phi$ is the quadratic character.
\end{theorem}

\section{\bf Preliminaries}

We start with a result which enables us to count the number of points on a curve using multiplicative characters
on $\mathbb{F}_p$ (see \cite[Proposition 8.1.5]{ireland}).
\begin{lemma}\label{lemma1}
Let $a\in\mathbb{F}_p^{\times}$. If $n|(p-1)$, then $$\#\{x\in\mathbb{F}_p: x^n=a\}=\sum \chi(a),$$
where the sum runs
over all characters $\chi$ on $\mathbb{F}_p$ of order dividing $n$.
\end{lemma}
Now we recall some standard facts regarding ordinary and Gaussian hypergeometric series. First, the ordinary
${_{2}}F_1$ hypergeometric series has the following integral representation \cite[p. 115]{erdelyi}:
\begin{align}\label{eq1}
{_{2}}F_1\left(\begin{array}{cccc}
                a, & b\\
                 & c
              \end{array}\mid x\right):=\frac{2\Gamma(c)}{\Gamma(b)\Gamma(c-b)}\int_0^{\pi/2}
              \frac{(\sin t)^{2b-1}(\cos t)^{2c-2b-1}}{(1-x\sin^2t)^a}dt,
\end{align}
where Re $c>$ Re $b>0$. Also, by \eqref{eq0} and \eqref{eq00} the Gaussian ${_{2}}F_1$ hypergeometric series over
$\mathbb{F}_p$ takes the form
\begin{align}\label{eq11}
{_{2}}F_1\left(\begin{array}{cccc}
                A, & B\\
                 & C
              \end{array}\mid x \right)=\epsilon(x)\frac{BC(-1)}{p}
              \sum_{y\in\mathbb{F}_p}B(y)\overline{B}C(1-y)\overline{A}(1-xy),
\end{align}
where $\epsilon$ denotes the trivial character.

Next we note two transformation properties of ordinary hypergeometric series.
Kummer's Theorem \cite[p. 9]{bailey} is given by
\begin{align}\label{eq2}
{_{2}}F_1\left(\begin{array}{cccc}
                a, & b\\
                 & 1+b-a
              \end{array}\mid -1 \right)=\frac{\Gamma(1+b-a)\Gamma(1+\frac{b}{2})}{\Gamma(1+b)\Gamma(1+\frac{b}{2}-a)},
\end{align}
while Pfaff's transformation \cite[p. 31]{slater} can be stated as
\begin{align}\label{eq3}
{_{2}}F_1\left(\begin{array}{cccc}
                a, & b\\
                 & c
              \end{array}\mid x \right)=(1-x)^{-a}{_{2}}F_1\left(\begin{array}{cccc}
                a, & c-b\\
                 & c
              \end{array}\mid \frac{x}{x-1} \right).
\end{align}
Greene \cite[p. 91]{greene} proved the following Gaussian analogs of these transformations:
\begin{align}\label{eq4}
{_{2}}F_1\left(\begin{array}{cccc}
                A, & B\\
                 & \overline{A}B
              \end{array}\mid -1 \right)=\left\{
                                                \begin{array}{ll}
                                                0, & \hbox{if $B$ is not a square;} \\
                                                \displaystyle \binom{C}{A}+
                                                \displaystyle \binom{\phi C}{A}, & \hbox{if $B=C^2$}
                                                 \end{array}
                                                 \right.
\end{align}
and
\begin{align}\label{eq5}
{_{2}}F_1\left(\begin{array}{cccc}
                A, & \overline{A}\\
                 & \overline{A}B
              \end{array}\mid \frac{1}{2} \right)=A(-2)\left\{
                                                \begin{array}{ll}
                                                0, & \hbox{if $B$ is not a square;} \\
                                                \displaystyle \binom{C}{A}+
                                                \displaystyle \binom{\phi C}{A}, & \hbox{if $B=C^2$,}
                                                 \end{array}
                                                 \right.
\end{align}
where $\phi$ is the quadratic character of $\mathbb{F}_p$.

\section{\bf Proof of the results}
\begin{pf}{\bf \ref{theorem3}.} Recalling \eqref{curve2}, from the definition of $\Omega(C_{l, \lambda})$, we have
\begin{align}
\Omega(C_{l, \lambda})&=2\int_a^b\frac{dx}{y^{l-1}}\nonumber \\
&=2\int_a^b\frac{dx}{2^{l-1}\{(x-a)(x-b)(x-c)\}^{\frac{l-1}{l}}}.\nonumber
\end{align}
Note that for $a<x<b$ and $0< \lambda <1$, $(x-a)$ is positive, while $(x-b)$ and $(x-c)$ are negative.
Hence $\Omega(C_{l, \lambda})$ is real.

Putting $(x-a)=(b-a)\sin^2\theta$, we obtain
\begin{align}
\Omega(C_{l, \lambda})&=2\int_0^{\pi/2}\frac{2(b-a)\sin\theta \cos\theta}{2^{l-1}
[(b-a)\sin^2\theta (b-a)\cos^2\theta\{(c-a)-(b-a)\sin^2\theta\}]^{\frac{l-1}{l}}}d\theta\nonumber\\
&=\frac{1}{2^{l-3}}\int_0^{\pi/2}\frac{(b-a)^{\frac{2-l}{l}}(\sin\theta)^{\frac{2-l}{l}}(\cos\theta)^{\frac{2-l}{l}}}
{\{(c-a)-(b-a)\sin^2\theta\}^{\frac{l-1}{l}}}d\theta.\nonumber
\end{align}
Using $(b-a)=\lambda$ and $(c-a)=1$ yields
\begin{align}
\Omega(C_{l,\lambda})&=\frac{1}{2^{l-3}\lambda^{\frac{l-2}{l}}}\int_0^{\pi/2}\frac{(\sin\theta)^{2\frac{1}{l}-1}
(\cos\theta)^{2\frac{2}{l}-2\frac{1}{l}-1}}{(1-\lambda \sin^2\theta)^{\frac{l-1}{l}}}d\theta\nonumber\\
&=\frac{(\Gamma(\frac{1}{l}))^2}{2^{l-2}\lambda^{\frac{l-2}{l}}
\Gamma(\frac{2}{l})}\cdot{_{2}}F_1\left(\begin{array}{cccc}
                (l-1)/l, & 1/l\\
                 & 2/l
              \end{array}\mid \lambda \right),\nonumber
\end{align}
where the last equality follows from \eqref{eq1}. This completes the proof of the theorem.
\end{pf}
\begin{remark}
If we put $l=2$ in Theorem \ref{theorem3}, we obtain Theorem \ref{theorem1}.
\end{remark}
\begin{remark}
As mentioned in Remark \ref{lem9}, $C_{3, \lambda}$ is isomorphic over $\mathbb{Q}$ to the elliptic
curve \eqref{curve3}. It would be interesting to know if there is any relation between $\Omega(C_{3, \lambda})$
and the real period of \eqref{curve3}.
\end{remark}
\begin{pf}{\bf \ref{theorem4}.} Since $p\equiv 1 ~(\text{mod}~l)$, there exists a character $\chi$ of order $l$
on $\mathbb{F}_p$. Using \eqref{eq11}, we have
\begin{align}
&\sum_{i=1}^{l-1}\chi^i(-\lambda^2){_{2}}F_1\left(\begin{array}{cccc}
                \overline{\chi^i}, & \chi^i\\
                 & \chi^{2i}
              \end{array}\mid \lambda \right)\notag\\
&=\sum_{i=1}^{l-1}\chi^i(-\lambda^2)\frac{\chi^i\chi^{2i}(-1)}{p}\sum_{t\in\mathbb{F}_p}\chi^i(t)
\overline{\chi^i}\chi^{2i}(1-t)\overline{\overline{\chi^i}}
(1-\lambda t)\nonumber\\
&=\sum_{i=1}^{l-1}\chi^i(-\lambda^2)\frac{\chi^{3i}(-1)}{p}\sum_{t\in\mathbb{F}_p}\chi^i(t)\chi^i(1-t)
\chi^i(1-\lambda t).\nonumber
\end{align}
Replacing $t$ by $\frac{t}{\lambda}$, we get
\begin{align}\label{04}
p\cdot\sum_{i=1}^{l-1}\chi^i(-\lambda^2){_{2}}F_1\left(\begin{array}{cccc}
                \overline{\chi^i}, & \chi^i\\
                 & \chi^{2i}
              \end{array}\mid \lambda \right)&=\sum_{i=1}^{l-1}\sum_{t\in\mathbb{F}_p}
              \chi^i(t(t-1)(t-\lambda))\nonumber\\
&=\sum_{t\in\mathbb{F}_p}\sum_{i=1}^{l-1}\chi^i(t(t-1)(t-\lambda)).
\end{align}
Moreover,
\begin{align}
&\#\{(x, y)\in \mathbb{F}_p^2: y^l=x(x-1)(x-\lambda)\}\nonumber\\
&=\sum_{t\in\mathbb{F}_p}\#\{y\in\mathbb{F}_p: y^l=t(t-1)(t-\lambda)\}\nonumber\\
&=\sum_{t\in\mathbb{F}_p, t(t-1)(t-\lambda)\neq0}
\#\{y\in\mathbb{F}_p: y^l=t(t-1)(t-\lambda)\}+\#\{t\in\mathbb{F}_p: t(t-1)(t-\lambda)=0\}.\nonumber
\end{align}
Now applying Lemma \ref{lemma1} and \eqref{04}, we obtain
\begin{align}
&\#\{(x, y)\in \mathbb{F}_p^2: y^l=x(x-1)(x-\lambda)\}\nonumber\\
&=\sum_{t\in\mathbb{F}_p}\sum_{i=0}^{l-1}\chi^i(t(t-1)(t-\lambda))+
\#\{t\in\mathbb{F}_p: t(t-1)(t-\lambda)=0\}\nonumber\\
&=p+\sum_{t\in\mathbb{F}_p}\sum_{i=1}^{l-1}\chi^i(t(t-1)(t-\lambda))\nonumber\\
&=p+p\cdot\sum_{i=1}^{l-1}\chi^i(-\lambda^2){_{2}}F_1\left(\begin{array}{cccc}
                \overline{\chi^i}, & \chi^i\\
                 & \chi^{2i}
              \end{array}\mid \lambda \right).\nonumber
\end{align}
Since ord$_p(\lambda(\lambda-1))=0$, using \eqref{eq45} we complete the proof of the result.
\end{pf}
\begin{remark}
Theorem \ref{theorem2} can be obtained from Theorem \ref{theorem4} by putting $l=2$. Note that for
the quadratic character $\phi$ of $\mathbb{F}_p$, we have $\phi(-\lambda^2)=\phi(-1)$.
\end{remark}
For $l\geq 3$, the genus of the curve $C_{l, \lambda}$ is $\frac{(l-1)(l-2)}{2}$. The Hasse-Weil bound
therefore yields the following result.
\begin{corollary}
Suppose $l\geq 4$. If $p\equiv 1$ $($\emph{mod} $l)$ and \emph{ord}$_p(\lambda(\lambda-1))=0$, then
$$\left |\sum_{i=1}^{l-1}\chi^i(-\lambda^2){_{2}}F_1\left(\begin{array}{cccc}
                \overline{\chi^i}, & \chi^i\\
                 & \chi^{2i}
              \end{array}\mid \lambda \right)\right|\leq \frac{(l-1)(l-2)}{\sqrt{p}},$$
              where $\chi$ is a character of $\mathbb{F}_p$ of order $l$.

If $l=3$, then $$ \left|2+p\cdot\sum_{i=1}^{2}\chi^i(-\lambda^2){_{2}}F_1\left(\begin{array}{cccc}
                \overline{\chi^i}, & \chi^i\\
                 & \chi^{2i}
              \end{array}\mid \lambda \right)\right|\leq 2{\sqrt{p}},$$
              where $\chi$ is a character of $\mathbb{F}_p$ of order $3$.
\end{corollary}
\begin{corollary}
If $p\equiv 1~($\emph{mod} $3)$ and $x^2+3y^2=p$, then
$$p\cdot\sum_{i=1}^{2}{_{2}}F_1\left(\begin{array}{cccc}
                \overline{\chi^i}, & \chi^i\\
                 & \chi^{2i}
              \end{array}\mid -1 \right)=(-1)^{x+y}\left(\frac{x}{3}\right)\cdot 2x-2,$$
              where $\chi$ is a character of $\mathbb{F}_p$ of order $3$.
\end{corollary}
\begin{proof}
As mentioned in Remark \ref{lem9}, $C_{3, -1}$ is isomorphic over $\mathbb{Q}$ to the elliptic curve
$y^2=x^3+1.$ By \cite[Proposition 2]{ono}, it is known that $a_p(C_{3, -1})=(-1)^{x+y-1}(\frac{x}{3})\cdot 2x$.
Now the result follows from Theorem \ref{theorem4}.
\end{proof}
\begin{remark}
The formula for $a_p(C_{3,\lambda})$ in Theorem \ref{theorem4} gives the trace of Frobenius of the family of
elliptic curves \eqref{curve3}.
\end{remark}
\begin{pf}{\bf \ref{theorem5}.} By \eqref{eq2}, we have
\begin{align}\label{for-1}
{_{2}}F_1\left(\begin{array}{cccc}
                (l-1)/l, & 1/l\\
                 & 2/l
              \end{array}\mid -1 \right)&=\frac{\Gamma(\frac{2}{l})\Gamma(\frac{2l+1}{2l})}{\Gamma(\frac{l+1}{l})
              \Gamma(\frac{3}{2l})}\nonumber\\
&=\frac{\frac{\Gamma(\frac{2l+1}{2l})}{\Gamma(\frac{l+1}{l})\Gamma(\frac{2l-1}{2l})}}
{\frac{\Gamma(\frac{3}{2l})}{\Gamma(\frac{2}{l})\Gamma(\frac{2l-1}{2l})}}\nonumber\\
&=\frac{\displaystyle \binom{\frac{1}{2l}}{\frac{1}{l}}}{\displaystyle \binom{\frac{3-2l}{2l}}{\frac{2-l}{l}}}.
\end{align}
Putting $\lambda=1/2$ in Theorem \ref{theorem3}, we obtain the relation
$$\frac{2^{\frac{l^2-3l+2}{l}}\Gamma(\frac{2}{l})}{(\Gamma(\frac{1}{l}))^2}\cdot\Omega(C_{l,\frac{1}{2}})
={_{2}}F_1\left(\begin{array}{cccc}
                (l-1)/l, & 1/l\\
                 & 2/l
              \end{array}\mid \frac{1}{2} \right).$$
Then using \eqref{eq3}, we find that
\begin{align}\label{eq46}
\frac{2^{\frac{l^2-3l+2}{l}}\Gamma(\frac{2}{l})}{(\Gamma(\frac{1}{l}))^2}\cdot
\Omega(C_{l,\frac{1}{2}})&=2^{\frac{l-1}{l}}{_{2}}F_1\left(\begin{array}{cccc}
                (l-1)/l, & 1/l\\
                 & 2/l
              \end{array}\mid -1 \right).
\end{align}
From \eqref{for-1} and \eqref{eq46}, we complete the proof of \eqref{eq1/2}.

Now, we shall prove the second part of the result. Note that $p$ is an odd prime.
Write $\chi=w^{k}$, where $w$ is a generator of the group of
Dirichlet characters mod $p$. Let $o(w)$ denote the order of $w$. Then $o(w)=p-1$
and $l=o(w^{k})=(p-1)/\text{gcd}(k,p-1)$. So $(p-1)/l=\text{gcd}(k,p-1)$. If $(p-1)/l$ is even,
then $k$ is also even, hence $\chi$ is a square. Conversely, if $\chi$ is a square,
it is an even power of the generator $w$, hence $k$ is even, and $(p-1)/l=\text{gcd}(k, p-1)$ is even.
This implies that $\chi$ is a square if and only if $(p-1)/l$ is even. Moreover, $\chi^i$ is a square
for even values of $i$, and for odd values of $i$, $\chi^i$ is a square if and only if $\chi$ is a square.
Using these, from Theorem \ref{theorem4} and \eqref{eq5}, we complete the proof of \eqref{eq1/22}.
\end{pf}
In \cite{rouse}, J. Rouse gave an analogy between ordinary hypergeometric series and Gaussian hypergeometric
series by evaluating $\Omega(C_{2, \frac{1}{2}})$ and $a_p(C_{2, \frac{1}{2}})$ in terms of hypergeometric series.
We now give an alternate proof of \cite[Theorem 3, p. 3]{rouse}.
\begin{pf}{\bf \ref{theorem6}.}
Putting $l=2$ in \eqref{eq1/2}, we obtain
\begin{align}
\frac{2^{-\frac{1}{2}}}{(\Gamma(\frac{1}{2}))^2}\cdot\Omega(C_{2, \frac{1}{2}})&
=\frac{{\displaystyle \binom{1/4}{1/2}}}{
\displaystyle \binom{-1/4}{0}}\nonumber
\end{align}
which yields
\begin{align}
\frac{\sqrt{2}}{2\pi}\cdot\Omega(C_{2, \frac{1}{2}})&={1/4 \choose 1/2},\nonumber
\end{align}
since ${-1/4 \choose 0}=1$ and $\Gamma(\frac{1}{2})=\sqrt{\pi}$.

For the second part, recall that $p\equiv 1$ (mod 4). Putting $l=2$ in \eqref{eq1/22}, we find that
$$\frac{-\phi(8)}{p}\cdot a_p(C_{2, \frac{1}{2}})=
\displaystyle \binom{\chi_4}{\phi}+\displaystyle \binom{\phi\chi_4}{\phi},$$
since $\chi_4^2=\phi$. Clearly $\phi\chi_4=\overline{\chi_4}$, and this implies that
$\displaystyle \binom{\phi\chi_4}{\phi}
=\overline{\displaystyle \binom{\chi_4}{\phi}}$.
Also, observing that $\phi(8)=\phi(2)$, we obtain
$$\frac{-\phi(2)}{2p}\cdot a_p(C_{2, \frac{1}{2}})=\text{Re}{\chi_4 \choose \phi}.$$
Since $p\equiv 1$ (mod 4), we have that $\phi(-1)=1$ and the result follows.
\end{pf}
Simplifying the expressions for $a_p(C_{l, \frac{1}{2}})$ given in Theorem \ref{theorem5},
we obtain the following result
which generalizes the case $l=2$, $p\equiv 1$ $($mod $4)$ treated in Theorem \ref{theorem6}.
\begin{corollary}\label{cor1}
Suppose that $p\equiv 1$ $($\emph{mod} $l)$. Then we have
\begin{align}
-a_p(C_{l, \frac{1}{2}})=\left\{\begin{array}{lllll}
2p\cdot\left[\phi(2)\emph{Re}\displaystyle{\chi_4\choose \phi}+
\displaystyle \sum_{i=1}^{\frac{l-4}{4}}\emph{Re}\left\{\chi^{-2i}(8)
\left(\displaystyle \binom{\chi^i}{\chi^{-2i}}+{\phi\chi^i \choose \chi^{-2i}}\right)\right\}\right],\\
&\hspace{-5.2cm} \hbox{if $\frac{p-1}{l}$ is odd and $l\equiv 0~($\emph{mod} $4)$;} \\
2p\cdot\displaystyle \sum_{i=1}^{\frac{l-2}{4}}\emph{Re}\left[\chi^{-2i}(8)
\left(\displaystyle \binom{\chi^i}{\chi^{-2i}}+{\phi\chi^i \choose \chi^{-2i}}\right)\right],\\
&\hspace{-5.2cm}\hbox{if $\frac{p-1}{l}$ is odd and $l\equiv 2~($\emph{mod} $4)$;}\\
2p\cdot\left[\phi(2)\emph{Re}\displaystyle{\chi_4\choose \phi}+
\displaystyle \sum_{i=1}^{\frac{l-2}{2}}\emph{Re}\left\{\psi^{-2i}(8)
\left(\displaystyle \binom{\psi^i}{\psi^{-2i}}+{\phi\psi^i \choose \psi^{-2i}}\right)\right\}\right],\\
&\hspace{-5.2cm} \hbox{if $\frac{p-1}{l}$ and $l$ are even;}\\
2p\cdot\displaystyle \sum_{i=1}^{\frac{l-1}{2}}\emph{Re}\left[\psi^{-2i}(8)
\left(\displaystyle \binom{\psi^i}{\psi^{-2i}}+{\phi\psi^i \choose \psi^{-2i}}\right)\right],\\
&\hspace{-5.2cm}\hbox{if $\frac{p-1}{l}$ is even and $l$ is odd, $l\geq 5$;}\\
2+2p\cdot \emph{Re}\left[\displaystyle {\chi \choose \chi}+ \displaystyle{\phi\chi \choose \chi}\right],
&\hspace{-5.2cm}\hbox{if $l=3$;}
\end{array}
\right.
\end{align}
where $\psi, \chi, \chi_4$ are characters of $\mathbb{F}_p$ of order $2l, l, 4$ respectively and
$\phi$ is the quadratic
character.
\end{corollary}
\begin{corollary}
If $p\equiv 1$ $($\emph{mod} $3)$ and $x^2+3y^2=p$, then
\begin{align}
p\cdot \emph{Re}\left[\displaystyle {\chi \choose \chi}+ \displaystyle{\phi\chi \choose \chi}\right]=(-1)^{x+y}
\left(\frac{x}{3}\right)\cdot x-1,\nonumber
\end{align}
where $\chi$ is a character of order $3$ on $\mathbb{F}_p$ and $\phi$ is the quadratic character.
\end{corollary}
\begin{proof}
As mentioned in Remark \ref{lem9}, $C_{3, -1}$ and $C_{3, \frac{1}{2}}$ are isomorphic over $\mathbb{Q}$ to the
elliptic curve $y^2=x^3+1.$ By \cite[Proposition 2]{ono}, it is known that
$a_p(C_{3, -1})=(-1)^{x+y-1}(\frac{x}{3})\cdot 2x$. From Corollary \ref{cor1}, we have
\begin{align}
-a_p(C_{3, \frac{1}{2}})=2+2p\cdot \text{Re}\left[\displaystyle {\chi \choose \chi}+
\displaystyle{\phi\chi \choose \chi}\right].\nonumber
\end{align}
Since $a_p(C_{3, -1})= a_p(C_{3, \frac{1}{2}})$, the result follows.
\end{proof}

\section*{Acknowledgment}
We thank Ken Ono for many helpful suggestions during the preparation of the article. We are grateful to
the referee for his/her helpful comments.

\end{document}